\documentclass[12pt]{amsart}
\usepackage{amsmath,amstext,amssymb, ulem, graphicx,multirow}

\newcommand{\sgn}{\operatorname{sgn}}
\newcommand{\Span}{\operatorname{span}}
\newcommand{\rank}{\operatorname{rank}}
\newcommand{\cl}{\operatorname{cl}}

\newcommand{\W}{\mathcal{W}}
\newcommand{\V}{\mathcal{V}}
\newcommand{\T}{\mathfrak{t}^+}
\newcommand{\s}{\mathfrak{s}^+}
\newcommand{\C}{\mathfrak{c}^+}
\newcommand{\g}{\mathfrak{g}}
\newcommand{\e}{\epsilon_0}
\newcommand{\ip}[2]{\left( #1 , #2 \right)}

\newtheorem*{theorem}{Theorem}
\newtheorem*{lemma}{Lemma}

\newtheorem*{cor}{Corollary}
\newtheorem*{remark}{Remark}

\title{Regular Orbital Measures on Lie Algebras}
\author{Alex Wright}
\address{Department of Pure Mathematics \\ University  of Waterloo \\ Waterloo ON Canada, N2L 3G1}
\email{amwright@uwaterloo.ca}
\thanks{This research was supported in part by NSERC}
\thanks{The author would like to thank Kathryn Hare for her generous
help in preparing this note.}

\begin{document}

  \baselineskip=17pt 

\begin{abstract}
Let $H_0$ be a regular element of an irreducible Lie Algebra $\g$,
and let $\mu_{H_0}$ be the orbital measure supported on $O_{H_0}$.
We show that $\hat{\mu}_{H_0}^k\in L^2(\g)$ if and only if
$k>\dim\g / (\dim\g-\rank\g)$.
\end{abstract}

\maketitle

\section{Introduction}

Let $G$ be a compact, connected, simple Lie group and $\g$ its Lie
algebra. It is well known that the non-trivial adjoint orbits in
$\g$ are compact submanifolds of proper dimension, but geometric
properties ensure that they generate $\g$. Consequently, if
$H_0\neq 0$ is in the torus $\mathfrak{t}$ of $\g$, and $\mu
_{H_0}$ is the orbital measure supported on the orbit $O_{H_0}$
containing $H_0$, ie, $\mu_{H_0}$ is the unique (up to
normalization) $G$-invariant measure on $O_{H_0}$, then some
convolution power
of $\mu _{H_0}$ is absolutely continuous to Lebesgue measure on $\g$ and even belongs to $L^{1+\varepsilon }$ for some $%
\varepsilon >0$ (see \cite{RS}). In \cite{Ra}, Ragozin showed that
dimension of $\g$ convolution powers sufficed, and this was
improved in a series of papers culminating in \cite{GHGAFA} with the minimal number of convolution powers being  $k_{G}=\rank G$ for the classical simple Lie algebras of type $%
B_{n}$, $C_{n}$ and $D_{n}$ and $k_{G}=\rank G+1$ for type
$A_{n}$. There it was also shown that if $\mu _{h}$ was the
orbital measure supported on the conjugacy class in $G$ containing
the non-central element $h$, then $\mu _{h}^{k_{G}}\in L^{2}(G)$.

In the simplest case $G=SU(2)$, $\g=\mathbb{R}^{3}$ and the
adjoint orbits are (the two dimensional) spheres centred at the
origin. The sum of two such spheres contains an open set and
consequently the convolution of any two orbital measures is
absolutely continuous \cite{Ra2}. In general, the generic orbits
(the so-called regular orbits defined below) have codimension
$\rank G$ and two convolution powers of such an orbital measure is
absolutely continuous (in either the group or algebra case).
Furthermore, for the generic orbital measure $\mu _{h}$ on the
group, one can use the Weyl character formula to see that $\mu
_{h}^{k_{2}}\in L^{2}(G)$ for $k_{2}=1+\rank G/(\dim G- \rank G)$
(see \cite{HStudia}) and this fact can be transferred to the Lie
algebra setting as well \cite{GHprivate}.

In this note we give a direct proof that if $\mu_{H_0}$ is any generic orbital measure on $\g$, then $\hat{\mu} _{H_0}^{k}\in L^{2}(\g)$ if and only if $%
k>1+\rank \g/(\dim \g-\rank \g)$. The novelty of our approach is
our geometric method, involving the root systems, of handling the
singularities which arise in the integral of the Fourier transform
of the measure.

Products of generic orbital measures are also studied in
\cite{DRW} and \cite {RT}; our approach recovers some of what was
proven in \cite{RT}.

\section{Definitions and Lemmas}

Let $T$ be a maximal torus of $G$ and $t$ be the corresponding
subalgebra of $\g$, also called the torus. Let $\Phi $ be the root
system of $\g$ with Weyl group $\W$ and positive roots $\Phi
^{+}$. Choose a base $\Delta =\{\beta _{1},\ldots  ,\beta _{n}\}$
for $\Phi$ and let $\T$ be the associated fundamental Weyl
chamber.
\begin{equation*}
\T=\{H\in t:(H,\beta _{j})>0\text{ for }j=1,\ldots  n\}
\end{equation*}
Given $H_0\in t$, the adjoint orbit of $H_0$ is given by
\begin{equation*}
O_{H_0}=\{Ad(g)H_0:g\in G\}\subseteq \g.
\end{equation*}
If $H_0\in \T$, then $H_0$ is called regular and $O_{H_0}$ is a
called a regular orbit.

The regular orbital measure, $\mu _{H_0}$, is the $G$-invariant
measure supported on the regular orbit $O_{H_0},$ normalized so
the Harish-Chandra formula gives
\begin{equation*}
\widehat{\mu }_{H_0}(H)=\frac{A_{H_0}(H)}{\prod_{\alpha \in \Phi ^{+}}(\alpha ,H)%
}\text{ for }H\in \T
\end{equation*}
where
\begin{equation*}
A_{H_0}(H)=\sum_{\sigma \in \W}\sgn(\sigma)
e^{i\ip{\sigma(H)}{H_0}}.
\end{equation*}

As $\mu _{H_0}$ is $G$-invariant, the Weyl integration formula implies that $%
\hat{\mu} _{H_0}^{k}\in L^{2}(\g)$ if and only if
\begin{equation*}
\int_{\T}\frac{\left| A_{H_0}(H)\right| ^{2k}}{\left|
\prod_{\alpha \in \Phi ^{+}}(\alpha ,H)\right| ^{2k-2}}dH<\infty
\text{.}
\end{equation*}
In this integral some of the inner products $(\alpha ,H)$
represent removable singularities on some walls of the Weyl
chamber. This is the primary obstacle in studying this integral
and we are able to deal with these singularities using geometry
and an induction argument.

Specifically, we will relate the integrand near a collection of
walls to the integrand for a subroot system. The power of our
induction is hidden in the fact that the the integrand is
continuous, and so is bounded on any neighborhood of the origin.
Several technical problems arise; in fact they are necessary
adaptations to the proof of a weaker result (Cor. 1), where the
technical results are not necessary. The case of a Lie algebra of
type $A_2$ is surprisingly representative, and the geometric
motivation for the results presented here come exclusively from
this case.

The notation will get slightly tedious, so we list it all here in
advance. Note that from now on we assume $\Phi$ is irreducible,
but we will consider reducible subroot systems of $\Phi$ that are
``simple"; these are simply those subroot systems for which a
subset of $\Delta$ can be chosen as a base.
\begin{center}
   \begin{tabular}{| l | l | }
     \hline
     $\g$ & An irreducible Lie algebra. \\ \hline
     $\Phi$ & The root system of $\g$. \\ \hline
     $\Delta=\{\beta_1,\ldots  ,\beta_n \}$ & The simple roots of
             $\Phi$. \\ \hline
     $\Phi^+$ & The positive roots of $\Phi$. \\ \hline
     $n$ & The rank of $\g$. \\ \hline
     $\W$ & The Weyl group of $\Phi$. \\ \hline
     $\T=\{H\in\g : \ip{H}{\beta_i}> 0 \text{ for all } i \}$ &
             The fund. Weyl chamber of $\g$. \\ \hline
     $\Psi$ & A simple subroot system of $\Phi$. \\ \hline
     $\V$ & The Weyl group of $\Psi$. \\ \hline
     $\{\gamma_1,\gamma_2,\ldots  ,\gamma_m \}\subset \Delta$ & A base
            for $\Psi$. \\ \hline
     $\Psi^+$ & The positive roots of $\Psi$. \\ \hline
     $m$ & Number of simple roots in $\Psi$. \\ \hline
     $ \s=\{ H\in \Span \Psi :\ip{H}{\gamma_i}> 0 \text{ for } i>1 \}$ &  The fund. Weyl
chamber of $\Psi$. \\ \hline
   \end{tabular}
\end{center}

Recall that every $H\in\s$ can be written as a non-negative linear
combination of the simple roots $\gamma_i$. This follows from the
fact that, in the irreducible case, all entries of the inverse of
the Cartan matrix are positive numbers. (See \cite{Hum}, section
13.4, exercise 8.)

We will need to break $\s$ up into the regions
$$R_i=\{H\in \s : \|H\|\geq 1, \ip{\gamma_i}{H}\geq
\ip{\gamma_j}{H} \text{ for all } j\}.$$
So $\s\setminus B_1=\cup_{i=1}^m R_i$. Now if
$$\Psi_1 = \Span_\mathbb{Z} \{\gamma_2,\ldots  ,\gamma_m\} \cap
\Psi$$
then the roots of $\Psi_1$ will correspond to removable
singularities on the walls of $\cl(R_1)$ when we calculate the
above integral with root system $\Psi$. Now let $\V_1$ be the Weyl
group of $\Psi_1$, and
$$\C=\{H\in \Span \Psi_1 : \ip{H}{\gamma_i} > 0 \text{ for }
i=2\ldots  m \}$$
be the fundamental Weyl chamber of $\Psi_1$. Finally, we define
$$P:\Span \Psi\to \Span \Psi : H \mapsto
\frac{1}{|\V_1|}\sum_{\sigma\in\V_1}\sigma(H).$$

\begin{lemma}[1]
Let $P$ be as above. Then
\begin{enumerate}
\item $\sigma(P(H))=P(H)$ for all $\sigma\in\V_1$.

\item $P$ is the projection from $\Span \Psi$ onto
$(\Span\Psi_1)^\perp$. So $I-P$ is the projection from $\Span
\Psi$ onto $\Span \Psi_1$.

\item $I-P$ in fact maps $\s$ to $\C$.

\item There are constants $a, b>0$ so that $\|P(H)\| \geq a
\|H\|$ and $\|(I-P)H\|\leq b \|P(H)\| $ if $H\in R_1$.
\end{enumerate}
\end{lemma}
Before reading the proof of this result, the reader is encouraged
to graphically verify part (ii) for the case $\Phi=A_2$.
\begin{proof}
(i) If $\sigma_1\in \V_1$ then
$$\sigma_1(P(H))=\frac{1}{|\V_1|}\sum_{\sigma\in\V_1}\sigma_1(\sigma(H))=\frac{1}{|\V_1|}\sum_{\sigma\in
\sigma_1\V_1}\sigma(H)=P(H).$$

(ii) Write $H=s+r$, where $s\in\Span\Psi_1$ and
$r\in(\Span\Psi_1)^\perp$. If $\alpha \in \Psi_1$ then
$$\sigma_\alpha(r)=r-\frac{2\ip{r}{\alpha}}{\ip{\alpha}{\alpha}}\alpha=r.$$
Since $\V_1$ is generated by reflections of the form
$\sigma_\alpha$, $\alpha\in\Psi_1$, it follows that $\sigma(r)=r$
for all $\sigma\in\V_1$. Hence $$P(H)=P(r)+P(s)=r+P(s).$$ If
$\alpha\in \Psi_1$ then by (i) $\sigma_\alpha(P(s))=P(s)$. Since
we also have
$$\sigma_\alpha(P(s))=P(s)-\frac{2\ip{P(s)}{\alpha}}{\ip{\alpha}{\alpha}}\alpha
$$ we get that $P(s)\in(\Span\Psi_1)^\perp$. But
$P(s)\in\Span\Psi_1$ so $P(s)=0$. Putting all of this together, we
get that $P(H)=r$ is the projection of $H$ onto $(\Span
\Psi_1)^\perp$.

Of course it follows that $H-P(H)$ is the projection of $H$ onto
$\Span\Psi_1$.

(iii) If $k>1$ and $H\in\s$ then $$\ip{\gamma_k
}{H-P(H)}=\ip{\gamma_k }{H}> 0$$ since
$P(H)\in\Span\{\gamma_2,\ldots  ,\gamma_m\}^\perp$.

(iv) Suppose, in order to obtain a contradiction, that $H\in
\cl(R_1)$ and $P(H)=0$. Then $H\in \Span\Psi_1$ and $H \in\cl(\s)$
so we can write $$H=c_2\gamma_2+\ldots  +c_m\gamma_m$$ with all
$c_i\geq 0$. Thus
$$\ip{H}{\gamma_1}=c_2\ip{\gamma_2}{\gamma_1}+\ldots
+c_m\ip{\gamma_m}{\gamma_1}.$$
We also have $\ip{\gamma_i}{\gamma_j}\leq 0$ if $i\neq j$, so in
fact $\ip{H}{\gamma_1}\leq 0$. Since $H\in \s$
$\ip{H}{\gamma_1}\geq 0$. Combining these we get
$\ip{H}{\gamma_1}=0$. From the definition of $R_1$ we get, for
each $i=1,\ldots, m$, that
$$0\leq \ip{H}{\gamma_i}\leq \ip{H}{\gamma_1} = 0$$
which contradicts the fact that $\|H\|\geq 1$. Thus $P(H)\neq 0$
on $\cl(R_1)$. In particular, $P(H)$ is nonzero on the compact set
$\cl(R_1) \cap \{H: \|H\|=1 \}$, so there is an $a>0$ such that
$a\leq \|P(H)\|$ if $\|H\|=1$, $H\in R_1$. Thus we see that
$a\|H\|\leq \|P(H)\|$ on $R_1$. Finally, we can take
$b=\frac1{a}+1$.
\end{proof}

We commented earlier that the roots of $\Psi_1$ will cause
problems in $R_1$ when integrating. As it turns out, all the other
roots of $\Psi$ are very well behaved on $R_1$. (It is quite
helpful to think of the roots of $\Psi_1$ as the ``good" roots on
$R_1$, and the roots of $\Psi\setminus \Psi_1$ as the ``bad"
roots.)

\begin{lemma}[2]
There exists $C>0$ such that for all
$\alpha\in\Psi^+\setminus\Psi_1^+$ and for all $H\in R_1$
$$\ip{H}{\alpha}\geq C\|H\|.$$
\end{lemma}
\begin{proof}
Take $\alpha\in\Psi^+\setminus\Psi_1^+$. Write $\alpha=\sum a_i
\gamma_i$ with all $a_i\geq 0$. Since $\alpha\notin \Psi_1$,
$a_1>0$. Now if $H\in \cl(R_1)$
$$ \ip{H}{\alpha}= \sum a_i\ip{H}{\gamma_i} \geq  a_1
\ip{H}{\gamma_1}>0.$$
Thus the function $$f(H)=\ip{H}{\alpha}$$ is non zero on the
compact set $\cl(R_1)\cap\{H:\|H\|=1\}$. Hence it attains a
positive minimum $M_\alpha$. We can take
$C=\min_{\alpha\in\Psi^+\setminus\Psi_1^+}M_\alpha$.
\end{proof}

We will be interested in subroot systems of $\Phi$ of the form
$$\{a_1\alpha_1+a_2\alpha_2+\ldots  +a_m\alpha_{m}:
a_i\in\mathbb{Z} \text{ for all } i\}\cap \Phi$$ where
$\{\alpha_1, \alpha_2,\ldots \alpha_{m}\}\subset \Delta$. We will
call these \textbf{simple} subroot systems. Note that $\Psi_1$ is
a simple subroot system of $\Phi$. Simple subroot systems are the
only type of subroot systems that will come up in our induction.
Restricting our attention to simple subroot systems makes the
verification of the following technical lemma easier.
\begin{lemma}[3]
Suppose $\Phi$ is an irreducible root system, with simple subroot
system $\Psi$ with $m$ simple roots, where $m<n$. Then
$$\frac{n}{|\Phi|}<\frac{m}{|\Psi|}$$
\end{lemma}
\begin{proof}
When we look at this result for a particular $m$, it is clearly
sufficient to prove it for the largest $\Psi$ with $m$ simple
roots. We list these subroot systems in Appendix A, along with the
ratios in question. See \cite{Hum} for basic facts needed about
subroot systems.
\end{proof}

It is worth noting that this lemma is not true if we allow $\Phi$
to be reducible. For example, consider a subroot system of Lie
type $B_3$ ($\frac{m}{|\Psi|}=\frac16$) in a root system of Lie
type $B_3\times A_1\times A_1\times A_1$
($\frac{n}{|\Phi|}=\frac14$).

We now set $\e>0$ to be any number with
$\e<\frac{m}{|\Psi|}-\frac{n}{|\Phi|}$ for all proper simple
subroot systems $\Psi$ of $\Phi$. We will need this $\e$ later for
technical reasons.

\section{The main result}

\begin{theorem}
Let $\Phi$ be an irreducible root system. Then
$\hat{\mu}_{H_0}^k\in L^2(\g)$ if and only if
$k>1+\frac{n}{|\Phi|}=\dim \g / (\dim \g - \rank\g)$.
\end{theorem}
\begin{cor}[1] If $\mu$ is a regular orbital measure, then
$\hat{\mu}^{\frac32}\in L^2(\g)$.
\end{cor}

\begin{cor}[2] If $\mu$ is a regular orbital measure then $\mu^2\in
L^p(\g)$ for all $p<\frac{\dim(\g)}{\rank(\g)}$.
\end{cor}
\begin{proof}
Our arguments show that $\hat{\mu}^2\in L^{p'}$ for
$p'<1+\frac{n}{|\Phi|}$. By the Hausdorff-Young inequality,
$\mu^2\in L^p$ for all $p<\frac{\dim(\g)}{\rank(\g)}$.
\end{proof}
It is worth noting that Corollary 2 is sharp when
$\g=\mathfrak{su}(2)$ by a result of Ragozin (see \cite{Ra2}, Prop
A.5).

\begin{proof}
(Of main theorem.) We prove a related result for all simple
subroot systems $\Psi$ of $\Phi$. All the notation will be as
before, including the definition of $\e$.

Our induction hypothesis: For all simple proper subroot systems
$\Psi$ of $\Phi$, if $k<1+\frac{n}{|\Phi|}+\e$ then $$\int_{\s\cap
B_r}
\frac{|A_{H_0}(H)|^{2k}dH}{|\prod{_{\alpha\in\Psi^+}\ip{\alpha}{
H}}|^{2k-2}} = O(r^{n-(k-1)|\Psi|}).$$ By this we mean that this
integral is bounded above, as a function of $r$, by
$Cr^{n-(k-1)|\Psi|}$ for some $C>0$.

For $m=1$, $\Psi$ is of Lie type $A_1$ and we get $$\int_1^r
\frac{|e^{itH_0}-e^{-itH_0}|^{2k}}{|t|^{2k-2}}dt.$$ Hence when
$k<1+\frac12$ the integrand is $O(r^{2-2k})$ and $2-2k>-1$. So if
$k<1+\frac12$ the integral is $O(r^{1-2(k-1)})$. Lemma 3 tells us
that
$1+\frac12=1+\frac{n}{|\Phi|}+(\frac12-\frac{n}{|\Phi|})>1+\frac{n}{|\Phi|}
+\e$.

Now we assume the result for all simple subroot systems of rank
$m-1$.

Consider a subroot system $\Psi$ of rank $m$. We will describe the
growth of the integral on $R_1$. Since we have not specified any
particular order among the $R_i$, and the integrand is continuous,
this is sufficient.

Let $\sigma_1, \ldots , \sigma_t$ be representatives from the left
cosets of $\V_1\leq \V$. We break up $|A_{H_0}|$ by cosets of
$\Psi_1$.

\begin{eqnarray*}
&&\int_{R_1\cap B_r} \frac{|
\sum_{j=1}^t\sum_{\sigma\in\V_1}\sgn(\sigma_j\sigma)
e^{i\ip{\sigma_j\sigma(H)}{H_0}}|^{2k}dH}{|\prod{_{\alpha\in\Psi^+}\ip{\alpha}{
H}}|^{2k-2}} \\ &\leq& 2^{2k}\sum_{j=1}^t \int_{R_1\cap B_r}
\frac{| \sum_{\sigma\in\V_1}\sgn(\sigma_j\sigma)
e^{i\ip{\sigma_j\sigma(H)}{H_0}}|^{2k}dH}{|\prod{_{\alpha\in\Psi^+}\ip{\alpha}{
H}}|^{2k-2}}\\
\end{eqnarray*}

For convenience we forget about the constant, and just write the
term of the $\sigma_j$ coset. We start by factoring out
$\left|\sgn(\sigma_j) e^{i\ip{\sigma_j(P(H))}{H_0}} \right|=1$ to
get $$ \int_{R_1\cap B_r} \frac{| \sum_{\sigma\in\V_1}\sgn(\sigma)
e^{i\ip{\sigma_j\sigma(H)}{H_0}-i\ip{\sigma_j(P(H))}{H_0}}|^{2k}dH}{|\prod{_{\alpha\in\Psi^+}\ip{\alpha}{
H}}|^{2k-2}}. $$

Since $P(H)=\sigma(P(H))$ (for $\sigma\in\V_1$) and
$\ip{\sigma(v)}{w}=\ip{v}{\sigma(w)}$ this integral equals $$
 \int_{R_1\cap B_r} \frac{| \sum_{\sigma\in\V_1}\sgn(\sigma)
e^{i\ip{\sigma(H-P(H))}{\sigma_j(H_0)}}|^{2k}dH}{|\prod{_{\alpha\in\Psi^+\setminus\Psi^+_1}\ip{\alpha}{
H}}|^{2k-2}|\prod{_{\alpha\in\Psi^+_1}\ip{\alpha}{ H}}|^{2k-2}}.
$$
Now we apply Lemma 2 to get the upper bound
$$ \int_{R_1\cap B_r}
\frac{C}{\|H\|^{(k-1)(|\Psi|-|\Psi_1|)}}\frac{|
\sum_{\sigma\in\V_1}\sgn(\sigma)
e^{i\ip{\sigma(H-P(H))}{\sigma_j(H_0)}}|^{2k}dH}{|\prod{_{\alpha\in\Psi^+_1}\ip{\alpha}{
H}}|^{2k-2}}.$$
At this point we can safely replace $\sigma_j(H_0)$ with
$H_0'=(I-P)\sigma_j(H_0)$.  Since $P(H)$ is orthogonal to
$\Psi_1$, we can change the inner products from $\ip{\alpha}{H}$
to $\ip{\alpha}{H-P(H)}$. If we also recall the bound
$\|P(H)\|\geq a\|H\|$ for all $H\in R_1$ from Lemma 1, this gives
$$ \int_{R_1\cap B_r}
\frac{C}{\|H\|^{(k-1)(|\Psi|-|\Psi_1|)}}\frac{|
\sum_{\sigma\in\V_1}\sgn(\sigma)
e^{i\ip{\sigma(H-P(H))}{H_0'}}|^{2k}dH}{|\prod{_{\alpha\in\Psi^+_1}\ip{\alpha}{
H-P(H)}}|^{2k-2}}$$
$$ \leq \int_{R_1\cap B_r}
\frac{C'}{\|P(H)\|^{(k-1)(|\Psi|-|\Psi_1|)}}\frac{|
\sum_{\sigma\in\V_1}\sgn(\sigma)
e^{i\ip{\sigma(H-P(H))}{H_0'}}|^{2k}dH}{|\prod{_{\alpha\in\Psi^+_1}\ip{\alpha}{
H-P(H)}}|^{2k-2}}.$$
$P$ maps onto a one dimensional subspace, say $\Span v_1$,
$\|v_1\|=1$. We can do a change of variables so that we are
integrating first with respect to $H'=H-P(H)\in \C$ and then $s$,
where $P(H)=sv_1$. If $a$ and $b$ are as in Lemma 1 then $s\geq a$
and $\|(I-P)H\|\leq b\|P(H)\|$ for $H\in R_1$ . Note that
$H\mapsto (P(H), (I-P)H)$ is an orthogonal change of variables so
the Jacobian is a constant.

If we now use Fubini's Theorem to rewrite our integral (and forget
the constant) we get
\begin{eqnarray*}
 &&
 \int_{a}^r \frac{1}{s^{(k-1)(|\Psi|-|\Psi_1|)}}\int_{\C\cap B_{b s}} \frac{| \sum_{\sigma\in\V_1}\sgn(\sigma)
e^{i\ip{\sigma(H')}{H'_0}}|^{2k}dH'}{|\prod{_{\alpha\in\Psi^+_1}\ip{\alpha}{
H'}}|^{2k-2}}ds.
\\
\end{eqnarray*}
 Note that no element of $\Psi$
annihilates $\sigma_j(H_0)$, so it is regular. It follows that no
element of $\Phi_1$ annihilates $H_0'$. Thus we can apply the
induction hypothesis. Since $m<n$,
$$1+\frac{m}{|\Psi|}=1+\frac{n}{|\Phi|}+(\frac{m}{|\Psi|}-\frac{n}{|\Phi|})>1+\frac{n}{\Phi}+\e
.$$ So if $k<1+\frac{n}{|\Phi|}+\e$ we have that the above
integral is at most
\begin{eqnarray*}
&&\int_{a}^r
\frac{1}{s^{(k-1)(|\Psi|-|\Psi_1|)}}O(s^{m-1-|\Psi_1|(k-1)})dr
\\
&=&O(r^{m-|\Psi|(k-1)}).
\\
\end{eqnarray*}
At some point in our induction we get that $n=m$ and $\Psi=\Phi$.
At this point our full induction hypothesis does not hold, but we
have that actual integral we are interested in is at most
$$\int_{\delta}^r O(r^{n-1-|\Phi|(k-1)})dr$$ if
$k<1+\frac{n}{|\Phi|}+\e$. This integral converges if
$$1+\frac{n}{|\Phi|}+\e>k>1+\frac{n}{|\Phi|}.$$

Hence $\hat{\mu}_{H_0}\in L^2(\g)$ if $k>1+\frac{n}{|\Phi|}$.

Now we show the necessity of the condition $k>1+\frac{n}{|\Phi|}$.

We can rewrite $$ \frac{|\sum_{\sigma\in\W}\sgn(\sigma)
e^{i\ip{\sigma(H)}{H_0}}|^{2k}}{|\prod{_{\alpha\in\Phi^+}\ip{\alpha}{
H}}|^{2k-2}}$$ as $$\frac1{\|H\|^{|\Phi|(k-1)}}
\frac{|\sum_{\sigma\in\W}\sgn(\sigma)
e^{i\|H\|\ip{\sigma(\frac{H}{\|H\|})}{H_0}}|^{2k}}
{\left|\prod{_{\alpha\in\Phi^+}\ip{\alpha}{\frac{H}{\|H\|}}}\right|^{2k-2}}$$
and consider this as $r^{-|\Phi|(k-1)} f(r, \phi_1, \ldots ,
\phi_{n-1})$, where $f$ is a function in polar coordinates.
$$f(r,\phi_1,\ldots  ,
\phi_{n-1})=\frac{|\sum_{\sigma\in\W}\sgn(\sigma)
e^{ir\ip{(1,\phi_1,\ldots  ,\phi_{n-1})}{H_0}}|^{2k}}
{|\prod{_{\alpha\in\Phi^+}\ip{\alpha}{(1,\phi_1,\ldots
,\phi_{n-1})}}|^{2k-2}}$$ As before, we will integrate in $\T$
with a ball around the origin removed, so we will always assume
$r\geq 1$.

If we fix $\Phi=(\phi_1,\ldots  ,\phi_{n-1})$, we see that
$f_\Phi(r):=f(r,\phi_1,\ldots  ,\phi_{n-1})$ is (the absolute
value of) the sum of continuous functions that are periodic in
$r$. Thus $f$ is almost periodic in $r$.

Since $\|\mu_{H_0}^k\|\neq 0$ and $f$ is continuous, we can find a
point $(r_0,\psi_1,\ldots  ,\psi_{n-1})$, a $\delta>0$ and a
$\epsilon>0$ so that if
$$U=\{(r, \phi_1,\ldots  ,\phi_{n-1}):
|\|\phi_i-\psi_i\|\leq\delta\forall i, |r-r_0|\leq\delta \}\subset
\T$$
then $f>2\epsilon>0$ on $U$.

We will have to change to polar coordinates to use this
observation. The Jacobian of this change of variables is
$$\Delta = r^{n-1}\sin^{n-2}\phi_1 \ldots  \sin^{n-2}\phi_{n-1}.$$
If necessary, we can modify $U$ so that $|\Delta|\geq Cr^{n-1}$ on
$U$, for some constant $C$. We then get that our integral greater
or equal to
$$\int_{\psi_1-\delta}^{\psi_1+\delta}\ldots
\int_{\psi_{n-1}-\delta}^{\psi_{n-1}+\delta}\int_{1}^\infty
C\frac1{r^{|\Phi|(k-1)}}f(r, \phi_1,\ldots  ,\phi_{n-1}) r^{(n-1)}
dr d\phi_{n-1}\ldots d\phi_1. $$

Say that $f_\Phi$ has an $\epsilon$ almost period in every
interval of size $M$. Pick $N\geq M+2\delta$. We know that
$f_\Phi(r)\geq 2\epsilon$ on $[r_0-\delta, r_0+\delta]$. Pick
$\tau_n$, an $\epsilon$ almost period of $f_\Phi$ in the interval
$[nN-r_0+\delta, (n+1)N-r_0-\delta]$, where $n>0$. Hence
$f_\Phi\geq\epsilon$ on $[r_0+\tau_n-\delta,
r_0+\tau_n+\delta]\subset [nN, (n+1)N]$. If $\chi_F$ is the
indicator function of $F=\bigcup_n [r_0+\tau_n-\delta,
r_0+\tau_n+\delta]$, then the inner integral is at least
$$\int_{1}^\infty C\epsilon \chi_F
r^{-(k-1)|\Phi|+n-1}dr\geq\int_{1}^\infty C\epsilon \chi_E
r^{-(k-1)|\Phi|+n-1}dr$$
where $E=\bigcup_n [nN, nN+2\delta]$. This integral is at least
$$\sum_{n=1}^\infty 2\delta C\epsilon (nN)^{-(k-1)|\Phi|+n-1}.$$
If $k\leq 1+\frac{n}{|\Phi|}$, this diverges. Thus the inner
integral is infinite for all $\phi_1,\ldots ,\phi_{n-1}$ in the
appropriate range. So if $k\leq 1+\frac{n}{|\Phi|}$, our integral
is infinite and $\mu^k\notin L^2(\g)$.
\end{proof}

\begin{remark}[1]
A similar result holds when $\Phi=\Phi_1\times\ldots \times\Phi_m$
is reducible. Say that the number of simple roots in $\Phi_i$ is
$r_i$, and the fundamental Weyl chamber of $\Phi_i$ is $\T_i$. In
this case the integrand splits to give
$$\int_{\T_1}\int_{\T_2}\ldots
\int_{\T_m}\frac{|A_{H_0}(t_1+t_2+\ldots
+t_m)|^{2k}}{|\prod{_{\alpha\in\Phi^+}\ip{\alpha}{ t_1+t_2+\ldots
+t_m}}|^{2k-2}}dt_m\ldots  dt_1.$$
This factors as
\begin{eqnarray*}
\int_{\T_1}\frac{A_{H_0}^{\Phi_1}(t_1)}{|\prod_{\alpha\in\Phi_1^+}\ip{\alpha}{t_1}}dt_1
\int_{\T_2}\frac{A_{H_0}^{\Phi_2}(t_2)}{|\prod_{\alpha\in\Phi_3^+}\ip{\alpha}{t_2}}dt_2
\ldots
\int_{\T_m}\frac{A_{H_0}^{\Phi_m}(t_m)}{|\prod_{\alpha\in\Phi_m^+}\ip{\alpha}{t_m}}dt_m.\\
\end{eqnarray*}
Since none of these factors can be zero, this is finite iff all
the integrals converge. Hence $\mu_{H_0}\in L^2(\g)$ iff $$k>\max
\{1+\frac{r_1}{|\Phi_1|}, 1+\frac{r_2}{|\Phi_2|},\ldots ,
1+\frac{r_m}{|\Phi_m|} \}.$$
\end{remark}

\begin{remark}[2] A measure $\mu$ is called $L^p$-improving if there is some
$p<2$ such that the operator $T_\mu: f \mapsto \mu*f$ is bounded
from $L^p(\g)$ to $L^2(\g)$. Using sophisticated arguments Ricci
and Travaglini \cite{RT} prove that for a regular, orbital measure
$\mu$, $T_\mu$ maps $L^p(\g)$ to $L^2(\g)$ if and only if $p\geq
1+\rank(\g)/(2\dim(\g)-\rank(\g))=p(\g)$. The same reasoning as
given in \cite{HStudia} Corollary 12 shows that our arguments give
the weaker result: $T_\mu$ is bounded from $L^p(\g)$ to $L^2(\g)$
for any $p>p(\g)$.
\end{remark}

\section*{Appendix A}
\begin{center}
   \begin{tabular}{| c | c | c | c | }
     \hline
     $\Phi$ & $\frac{n}{|\Phi|}$ & $\Psi$ & $\frac{m}{|\Psi|}$  \\
     \hline\hline
     $A_n$ & $\frac1{n+1}$ & $A_{m}, m<n$ & $\frac1{m+1}$  \\ \hline
     $B_n$ & $\frac1{2n}$ & $B_{m}, m<n$ & $\frac1{2m}$  \\ \hline
     $C_n$ & $\frac1{2n}$ & $C_{m}, m<n$ & $\frac1{2m}$  \\ \hline
     $D_n$ & $\frac1{2(n-1)}$ & $D_{m}, m<n$ & $\frac1{2(m-1)}$  \\ \hline
     $E_6$ & $\frac1{12}$ & $D_{m},m<6$ & $\frac1{2(m-1)}$  \\ \hline
     $E_7$ & $\frac1{18}$ & $D_{m},m<7$ & $\frac1{2(m-1)}$  \\ \hline
     $E_7$ & $\frac1{18}$ & $E_6$ & $\frac1{12}$  \\ \hline
     $E_8$ & $\frac1{30}$ & $D_{m},m<8$ & $\frac1{2(m-1)}$  \\ \hline
     $E_8$ & $\frac1{30}$ & $E_6$ & $\frac1{12}$  \\ \hline
     $E_8$ & $\frac1{30}$ & $E_7$ & $\frac1{18}$  \\ \hline
     $F_4$ & $\frac1{12}$ & $B_{m},m<4$ & $\frac1{2m}$  \\ \hline
     $G_2$ & $\frac1{6}$ & $A_1$ & $\frac1{2}$  \\ \hline
   \end{tabular}
 \end{center}

\end{document}